\documentclass[11pt]{amsart}

\usepackage{latexsym, amssymb,enumerate}
\usepackage{appendix}

\newcommand{\be}{\begin{equation}}
\newcommand{\ee}{\end{equation}}
\newcommand{\beq}{\begin{eqnarray}}
\newcommand{\eeq}{\end{eqnarray}}

\def\H{{\mathbb H}}

\def\R{{\mathfrak R}}

\newtheorem{prop}{Proposition}[section]
\newtheorem{thm}[prop]{Theorem}
\newtheorem{lemm}[prop]{Lemma}
\newtheorem{coro}[prop]{Corollary}
\newtheorem{rema}[prop]{Remark}

\newtheorem{conj}[prop]{Conjecture}

\def\begeq{\begin{equation}}
\def\endeq{\end{equation}}

\def\R{\mathbb R}

\def\d{\delta}

\def\s{\sigma}
\def\l{\lambda}

\def\S{{\mathbb S}}

\def \ds{\displaystyle}
\def \vs{\vspace*{0.1cm}}

\def\odot{\setbox0=\hbox{$\bigcirc$}\relax \mathbin {\hbox
to0pt{\raise.5pt\hbox to\wd0{\hfil $\wedge$\hfil}\hss}\box0 }}

\numberwithin{equation} {section}

\begin{document}

\title[Hyperbolic Alexandrov-Fenchel inequalities II] {Hyperbolic Alexandrov-Fenchel quermassintegral inequalities II}

\author{Yuxin Ge}
\address{Laboratoire d'Analyse et de Math\'ematiques Appliqu\'ees,
CNRS UMR 8050,
D\'epartement de Math\'ematiques,
Universit\'e Paris Est-Cr\'eteil Val de Marne, \\61 avenue du G\'en\'eral de Gaulle,
94010 Cr\'eteil Cedex, France}
\email{ge@u-pec.fr}

\author{Guofang Wang}
\address{ Albert-Ludwigs-Universit\"at Freiburg,
Mathematisches Institut
Eckerstr. 1
D-79104 Freiburg}
\email{guofang.wang@math.uni-freiburg.de}

\author{Jie Wu}
\address{School of Mathematical Sciences, University of Science and Technology
of China Hefei 230026, P. R. China
\and
 Albert-Ludwigs-Universit\"at Freiburg,
Mathematisches Institut
Eckerstr. 1
D-79104 Freiburg
}
\email{jie.wu@math.uni-freiburg.de}

\thanks{The first named author  is partly supported by ANR  project
ANR-08-BLAN-0335-01. The  second and third named authors are partly supported by SFB/TR71
``Geometric partial differential equations''  of DFG}

\begin{abstract}In this paper we first establish an optimal Sobolev type inequality for hypersurfaces in $\H^n$(see Theorem \ref{mainthm1}). As an application we obtain
hyperbolic  Alexandrov-Fenchel inequalities for curvature integrals and quermassintegrals.
Precisely, we prove a following geometric inequality in the hyperbolic space $\H^n$, which is a hyperbolic  Alexandrov-Fenchel inequality,
\begin{equation*}
\begin{array}{rcl}
\ds \int_\Sigma \s_{2k}\ge \ds\vs C_{n-1}^{2k}\omega_{n-1}\left\{ \left( \frac{|\Sigma|}{\omega_{n-1}} \right)^\frac 1k +
\left( \frac{|\Sigma|}{\omega_{n-1}} \right)^{\frac 1k\frac {n-1-2k}{n-1}} \right\}^k,
\end{array}
\end{equation*}
provided that $\Sigma$ is a horospherical convex, where $2k\leq n-1$. Equality holds if and only if $\Sigma$ is a geodesic sphere in $\H^n$.
Here $\sigma_{j}=\s_{j}(\kappa)$ is the $j$-th mean curvature and $\kappa=(\kappa_1,\kappa_2,\cdots, \kappa_{n-1})$ is the set of the principal curvatures  of $\Sigma$.
Also, an optimal inequality for quermassintegrals in $\H^n$ is as following:
$$
W_{2k+1}(\Omega)\geq\frac {\omega_{n-1}}{n}\sum_{i=0}^k\frac{n-1-2k}{n-1-2k+2i}\,C_k^i\bigg(\frac{nW_1(\Omega)}{\omega_{n-1}}\bigg)^{\frac{n-1-2k+2i}{n-1}},
$$
provided that $\Omega\subset\H^n$ is a domain with $\Sigma=\partial\Omega$  horospherical convex, where $2k\leq n-1$. Equality holds if and only if $\Sigma$ is a geodesic sphere in $\H^n$. Here $W_r(\Omega)$ is quermassintegrals in integral geometry. 
\end{abstract}

\maketitle

\section{Introduction}

In this paper we first establish  Sobolev type inequalities for hypersurfaces in the hyperbolic space $\H^n$.
Let $g$ be a Riemannian metric on a Riemannian manifold. Its $k$th Gauss-Bonnet curvature (or Lovelock curvature) $L_k$ is defined by (see \cite{GWW} for instance)
 \begin{equation}\label{Lk}
L_k:=\frac{1}{2^k}\d^{i_1i_2\cdots i_{2k-1}i_{2k}}
_{j_1j_2\cdots j_{2k-1}j_{2k}}{R_{i_1i_2}}^{j_1j_2}\cdots
{R_{i_{2k-1}i_{2k}}}^{j_{2k-1}j_{2k}}.
\end{equation}
Here ${R_{ij}}^{kl}$ is the Riemannian curvature with respect to $g$, and  the generalized Kronecker delta is defined by
\[
 \d^{j_1j_2 \dots j_r}_{i_1i_2 \dots i_r}=\det\left(
\begin{array}{cccc}
\d^{j_1}_{i_1} & \d^{j_2}_{i_1} &\cdots &  \d^{j_r}_{i_1}\\
\d^{j_1}_{i_2} & \d^{j_2}_{i_2} &\cdots &  \d^{j_r}_{i_2}\\
\vdots & \vdots & \vdots & \vdots \\
\d^{j_1}_{i_r} & \d^{j_2}_{i_r} &\cdots &  \d^{j_r}_{i_r}
\end{array}
\right).
\]
When $k=1$, $L_1$ is just the scalar curvature $R$.
When $k=2$, it is the  so-called  (second) Gauss-Bonnet curvature
$$
L_2=\|Rm\|^2-4\|Ric\|^2+R^2,
$$
where $Rm$,  $Ric$ are  the Riemannian curvature tensor, and the Ricci  tensor with respect to $g$ respectively.
The Gauss-Bonnet curvature $L_k$ is a very natural generalization of the scalar curvature.
When the underlying manifold is local conformally flat,
$L_k$ equals to  the $\s_k$-scalar curvature up to a constant multiple, precisely(cf. \cite{GWW3})
 \beq\label{L&sigma}
 L_k=2^kk!(n-1-k)(n-2-k)\cdots(n-2k)\s_k(g).
 \eeq
 Here the $\s_k$-scalar curvature was introduced in Viaclovsky \cite{Via}  by
\beq\label{sigma}
\s_k(g):=\s_k (\Lambda_g),
\eeq
 and $\Lambda_g$ is the set of the eigenvalues of the Schouten tensor $A_g$ defined by
\begin{equation}\label{Schouten}
 A_g=\frac 1{n-3}\left(Ric_g-\frac {R_g}{2(n-2)} g\right).
\end{equation}
Here we consider the $(n-1)$-dimensional manifold $M$ with metric $g$. The $\s_k$-scalar curvature is also a very  natural generalization of the scalar curvature $R$ (in fact, $\s_1(g)=\frac 1{2(n-2)} R$) and  has been
intensively studied in the fully nonlinear Yamabe problem. The fully nonlinear Yamabe problem for $\s_k$ is a generalization of ordinary Yamabe problem for the scalar curvature $R$.
 In the ordinary Yamabe problem, the following functional, the so-called Yamabe functional,  plays a crucial role
\beq\label{eq1}
{\mathcal F}_1(g)=(vol(g))^{-\frac{n-3}{n-1}}\int R_g d\mu(g).
\eeq
For a given conformal class $[g]=\{e^{-2u} g \,|\, u\in C^\infty (M)\}$, the Yamabe constant is defined
by
\[Y_1([g])=\inf _{\tilde g \in [g]} {\mathcal F}_1(\tilde g).\]
By the resolution of the Yamabe problem, Aubin and Schoen \cite{Aubin,Schoen} proved that for any metric $g$ on $M$
\beq
\label{eq2}
\quad Y_1([g]) \le Y_1([g_{\S^{n-1}}]) \quad\hbox{ and } \;Y_1([g]) < Y_1([g_{\S^{n-1}}]) \hbox{ for any } (M,[g]) \hbox{ other than } [g_{\S^{n-1}}],\eeq
where $[g_{\S^{n-1}}]$ is the conformal class of the standard round metric on the sphere $\S^{n-1}$. From this, one can see the importance of the constant
$ Y_1([g_{\S^{n-1}}])$. In fact, one can prove that
\beq\label{eq3}
Y_1([g_{\S^{n-1}}])=(n-1)(n-2)\omega_{n-1}^{\frac 2{n-1}},
\eeq
where $\omega_{n-1}$ is the volume of $g_{\S^{n-1}}$. It is trivial to see that  \eqref{eq3}  is equivalent to
\beq\label{eq4}
\int_M L_1 d\mu(g)=\int_M R_gd\mu(g) \ge (n-1)(n-2)\omega_{n-1}^{\frac 2{n-1}} vol(g)^{\frac {n-3}{n-1}},
\eeq
$\hbox{for any } g\in [g_{\S^{n-1}}]$, which is in fact an optimal Sobolev inequality. See \cite{Hebey}.
As a natural generalization, we proved in \cite{GuanWang} a generalized  Sobolev inequality for $\s_k$-scalar curvature $\s_k(g)$,
which states
\beq\label{eq4_1}
\int_M \s_k(g)d\mu(g) \ge \frac{C_{n-1}^k}{2^k}\omega_{n-1}^{\frac {2k}{n-1}} vol(g)^{\frac {n-1-2k}{n-1}},
\eeq
$\hbox{for any } g\in {\mathcal C}_{k-1} ([g_{\S^{n-1}}]),$
where $ {\mathcal C}_{k-1} ([g_{\S^{n-1}}])=[g_{\S^{n-1}}]\cap \Gamma^+_{k-1}$ and $\Gamma^+_{k}=\{g\,|\, \s_{j}(g)>0, \forall j\le k\}$. In this paper, we denote $C_{n-1}^k=\frac{(n-1)!}{k!(n-1-k)!}.$
By  (\ref{L&sigma}) inequality \eqref{eq4_1} can be written  in the following form  
\beq\label{eq6}
\int_\Sigma L_k d\mu(g) \ge  C_{n-1}^{2k}(2k)!\;\omega_{n-1}^{\frac{2k}{n-1}}(vol(g))^{\frac{n-1-2k}{n-1}},
\eeq
 for any  $g\in {\mathcal C}_{k-1} ([g_{\S^{n-1}}]).$
We call both inequalities \eqref{eq4}, \eqref{eq6}   {\it optimal Sobolev inequalities}
 and would like to investigate which classes of metrics satisfy the optimal Sobolev inequalities.
\eqref{eq4} and \eqref{eq6} mean that a suitable subclass of the conformal class of the standard round metric  satisfies the optimal Sobolev inequalities.
From \eqref{eq2} we know in any conformal class other than the conformal class of the standard round metric, there exist many metrics which do not satisfy
the optimal  Sobolev inequality. Hence it is natural to ask if there are other interesting classes of metrics satisfy the  optimal Sobolev inequality?
Observe that for a closed hypersurface $\Sigma$ in $\R^n$,
\beq\label{R}L_k=(2k)!\s_{2k},\eeq
where $\s_{2k}$ is the $2k$-mean curvature of $\Sigma$, which is defined by
\[\s_j=\s_j(\kappa),\]
where $\kappa=(\kappa_1,\kappa_2,\cdots, \kappa_{n-1})$, $\kappa_j$ ($1\le j\le n-1$) is the principal curvature of $B$,
and $B$ is the 2nd fundamental form of $\Sigma$ induced by the standard Euclidean metric.
 The classical Alexandrov-Fenchel inequality (see \cite{Schneider} for instance) implies for convex hypersurfaces in $\R^n$ that
\beq\label{eq7}
\int_\Sigma L_k(g) d\mu(g)=(2k)!\int_\Sigma \s_{2k} d\mu(g) \ge C_{n-1}^{2k}(2k)!\;\omega_{n-1}^{\frac{2k}{n-1}} |\Sigma|^{\frac{n-1-2k}{n-1}}.
\eeq
I this paper we use $|\Sigma|$  to denote  the area of $\Sigma$ with respect to the induced metric.
Inequality \eqref{eq7} means that the induced metric of any convex hypersurfaces  in  $\R^n$ satisfy the optimal Sobolev inequalities. The convexity can be weakened. See the work of Guan-Li \cite{GuanLi}, Huisken
\cite{Huisken} and Chang-Wang \cite{ChangWang}.

In this paper we prove that the induced metric of
horospherical convex hypersurfaces in $\H^n$ also satisfy the optimal Sobolev inequalities.

\begin{thm}\label{mainthm1}Let $2k< n-1$. Any horospherical convex hypersurfaces $\Sigma$ in $\H^n$ satisfies
\beq\label{eq8}
\int_\Sigma L_k d\mu(g) \ge  C_{n-1}^{2k}(2k)!\;\omega_{n-1}^{\frac{2k}{n-1}}|\Sigma|^{\frac{n-1-2k}{n-1}},
\eeq
equality holds if and only if $\Sigma$ is a geodesic sphere.
\end{thm}
A hypersurface in $\H^n$ is {\it horospherical convex}
 if all principal curvatures are larger than or equal to $1$.
 The horospherical convexity is a natural geometric concept, which is equivalent to
 the geometric convexity in Riemannian manifolds. For any hypersurface in $\H^n,$ the Gauss-Bonnet curvature $L_k$ of the induced metric of the hypersurface  can be expressed in terms of
 the curvature integrals by (see also Lemma \ref{lem0} below)
 \beq\label{H}L_k=C_{n-1}^{2k}(2k)! \sum_{j=0}^k (-1)^j\frac{C_k^j}{C_{n-1}^{2k-2j}}\s_{2k-2j}.\eeq
 Comparing \eqref{eq7} for $\R^n$ with \eqref{eq8} for $\H^n$ and \eqref{R} with  \eqref{H}, we obtain the same inequality for
 $L_k$, while $L_k$ has diferent expression in terms of the curvature integrals.
We remark that when $2k=n-1$, \eqref{eq8} is an equality for any hypersurface diffeomorphic to a sphere, i.e,
$$\int _\Sigma L_{\frac{n-1}2} dv(g) = (n-1)!\omega_{n-1}.$$ This follows that the Gauss-Bonnet-Chern theorem.

As a first direct application, we establish  Alexandrov-Fenchel type inequalities for curvature integrals.
\begin{thm}
\label{mainthm2} Let $2k\le n-1$. Any horospherical convex hypersurface  $\Sigma\subset\H^n$ satisfies
\begin{equation}\label{AFk}
\begin{array}{rcl}
\ds \int_\Sigma \s_{2k}\ge \ds\vs C_{n-1}^{2k}\omega_{n-1}\left\{ \left( \frac{|\Sigma|}{\omega_{n-1}} \right)^\frac 1k +
\left( \frac{|\Sigma|}{\omega_{n-1}} \right)^{\frac 1k\frac {n-1-2k}{n-1}} \right\}^k,
\end{array}
\end{equation}
equality holds if and only if $\Sigma$ is a geodesic sphere.
\end{thm}

When $k=1$ Theorem \ref{mainthm1}, and hence Theroem \ref{mainthm2},
is true even for any star-shaped and two-convex hypersurfaces in $\H^n$, ie., $\sigma_1\ge 0$ and $\sigma_2\ge 0$, which was proved by
Li-Wei-Xiong in a recent work \cite{LWX}. When $k=2$, Theorem \ref{mainthm1}  was proved in our recent paper \cite{GWW_AF2}.
Due to the complication of the variational structure of $\int \s_k$
in the hyperbolic space, the case $k\geq 2$ is quite different from  the case $k=1$. For case $k\geq 2$ the horospherical convexity of the hypersurface $\Sigma$ plays an essential role.

   At the end of this paper we show that a similar inequality holds for $\sigma_1$ and propose a conjecture for general odd $\sigma_{2k+1}$.

  Another application is an optimal inequality for quermassintegrals in $\H^n$.
For a (geodesically) convex domain $\Omega \subset \H^n$ with $\Sigma=\partial\Omega$, quermassintegrals are defined by
\beq\label{Wdef.}
W_r(\Omega) :=\frac{(n-r)  \omega_{r-1} \cdots\omega_0}
{n \omega_{n-2}\cdots    \omega_{n-r-1}}\int_{{\mathcal L}_r} \chi(L\cap \Omega) dL,
\eeq
where ${\mathcal L}_r$ is the space of $r$-dimensional totally geodesic subspaces $L$ in $\H^n$,
$\omega_r$ is the area of the $r$-dimensional standard round sphere
and $dL$ is the natural (invariant) measure on  ${\mathcal L}_r$  (cf. \cite{Santalo}, \cite{Solanes}). As in the Euclidean
case we take $W_0(\Omega) = Vol(\Omega)$. With these definitions, unlike the euclidean case,
 the quermassintegral in $\H^n$ do not coincide with the  mean curvature integrals,
but they are closely related (cf. \cite{Solanes})
\begin{equation}\label{relation1}
\frac 1{C_{n-1}^r} \int_{\Sigma} \sigma_r = n\left(W_{r+1}(\Omega) +\frac{r}
{n- r + 1}
W_{r-1}(\Omega)\right),  \quad W_0(\Omega)=Vol(\Omega), \quad  W_1(\Omega)=\frac 1n |\Sigma|.
\end{equation}
The relationship between $W_0$ and $W_1$, the hyperbolic isoperimetric inequality,  was established by Schmidt \cite{Schmidt} 70 years ago.
When $n=2$, the  hyperbolic isoperimetric inequality
is
\[ L^2\ge 4\pi A+A^2,\]
where $L$ is the length of a  curve $\gamma$ in $\H^2$ and $A$ is the area of the enclosed domain by $\gamma$. In general,
 this hyperbolic isoperimetric inequality has no explicit form.
 There are many attempts to establish  relationship between $W_k(\Omega)$ in the
hyperbolic space $\H^n$. See,  for example, \cite{Santalo} and \cite{Solanes_Thesis}. In \cite{GS}, Gallego-Solanes proved by using integral geometry
the following interesting inequality for convex domains in $\H^n$, precisely, there holds,
\beq \label{gs} W_r(\Omega)>\frac{n-r}{n-s}W_s(\Omega), \quad r>s,\eeq
which implies
\beq\label{gs2}
\int_\Sigma \s_k d\mu  > cC_{n-1}^k |\Sigma|,
\eeq
where $c=1$ if $k>1$ and $c=(n-2)/(n-1)$ if $k=1$ and $|\Sigma|$ is the area of $\Sigma$. Here $d\mu$ is the area element of the induced metric.
The constants in \eqref{gs} and \eqref{gs2} are optimal in the sense that one can not replace them by bigger constants. However, they are far away being optimal.

As another application of Theorem \ref{mainthm1}, we have the following optimal inequalities of $W_{k}(\Omega)$ for general odd $k$ in terms of $W_1=\frac 1  n |\Sigma|$.

\begin{thm}\label{Wkestimate} Let $2k\le n-1$. If $\Omega\subset \H^n$ be a domain with $\Sigma=\partial \Omega$ horospherical convex, then
\begin{equation}\label{W2k+1}
W_{2k+1}(\Omega)\geq\frac {\omega_{n-1}}{n}\sum_{i=0}^k\frac{n-1-2k}{n-1-2k+2i}\,C_k^i\bigg(\frac{nW_1(\Omega)}{\omega_{n-1}}\bigg)^{\frac{n-1-2k+2i}{n-1}},
\end{equation}
where $\omega_{n-1}$ is the area of the unit sphere $\S^{n-1}$.
Equality holds if and only if $\Sigma$ is a geodesic sphere.
\end{thm}

As a direct corollary we solve an isoperimetric problem for horospherical convex surfaces with fixed $W_1$.

\begin{coro}\label{Coro1}  Let $2k\le n-1$. In a class of horospherical convex hypersurfaces in $\H^n$ with fixed $W_1$,
the minimum of $W_{2k+1}$ is achieved by and only by the geodesic spheres.
\end{coro}

Corollary \ref{Coro1} answers a question  asked in the paper of Gao, Hug and Schneider \cite{GHS} in this case.

\

In order to prove Theorem \ref{mainthm1}, motivated by \cite {GWW_AF2}  and \cite{LWX} (see also \cite{BHW} and \cite{deLG}), we consider the following
functional
\begin{equation}\label{func}
Q(\Sigma):=|\Sigma|^{-\frac{n-1-2k}{n-1}} \int_{\Sigma}L_k.
\end{equation}
Here $L_k$  is the Gauss-Bonnet curvature with respect to the induced metric $g$ on $\Sigma$.
This is a Yamabe type functional. One of crucial points of this paper is to show that functional $Q$
is non-increasing under the following inverse curvature flow
\begin{equation}\label{flow0}
\frac{\partial \Sigma_t}{\partial t}= \frac{n-2k}{2k}\frac{\s_{2k-1}}{\s_{2k}}\nu,
\end{equation}
where $\nu$ is  the outer normal of $\Sigma_t$, provided that the initial hypersurface is horospherical convex. One can show that horospherical convexity is
preserved by flow \eqref{flow0}. By the convergence results of Gerhardt \cite{Gerhardt} on the inverse  curvature flow \eqref{flow0},  we  show that the flow approaches to surfaces
whose induced metrics belong to the conformal class of the standard round sphere metric. Therefore, we can use the result \eqref{eq6} to
\[ Q(\Sigma)\ge \lim_{t\to \infty} Q(\Sigma_t) \ge C_{n-1}^{2k}(2k)!\omega_{n-1}^{\frac{2k}{n-1}}.\]

\vspace{3mm}

The rest of this paper is organized as follows. In Section 2 we present some basic facts about the elementary functions $\sigma_k$
and recall the  generalized Sobolev inequality \eqref{eq6} from \cite{GuanWang}. In Section 3, We present the relationship between various geometric quantities including the intrinsic geometric quantities $\int_{\Sigma}L_k$, the curvature integrals $\int_{\Sigma}\s_k$ and the quermassintegrals $W_r(\Omega)$ .   In Section 4 we prove the crucial monotonicity of $Q$
and analyze  its asymptotic behavior under flow \eqref{flow0}. The proof of our main theorems are given in Section 5. In Section 6, we show that a similar inequality holds for $\s_1$ and propose a conjecture for integral integrals  $\s_{2k+1}$.

\section{Preliminaries}

Let $\s_k$ be the $k$-th elementary symmetry function $\s_k:\R^{n-1}\to \R$ defined by
\[\s_k(\Lambda)=\sum_{i_1<\cdots<i_{k}}\l_{i_1}\cdots\lambda_{i_k}\quad  \hbox{ for } \Lambda=(\l_1, \cdots,\l_{n-1})\in \R^{n-1}.\]
 For a symmetric matrix $B$, denote $\lambda(B)=(\lambda_1(B),\cdots,\lambda_n(B))$ be the eigenvalues of $B$. We set
\[
\s_k(B):=\s_k(\lambda(B)).
\]
 The Garding cone $\Gamma_k^+$ is defined as
\[\Gamma_k^+=\{\Lambda \in \R^{n-1} \, |\,\s_j(\Lambda)>0, \quad\forall j\le k\}.\]
A symmetric matrix $B$ is called belong to $\Gamma_k^+$  if $\lambda(B)\in \Gamma_k^+$.
We collect the basic facts about $\s_k$, which will be directly used in this paper. For other related facts, see a survey of Guan \cite{Guan} or \cite{LWX}.
\begin{equation}\label{sigmak}
\s_k(B) =\ds\frac{1}{k!}\d^{i_1\cdots i_k}
_{j_1\cdots j_k}b_{i_1}^{j_1}\cdots
{b_{i_k}^ {j_k}},
\end{equation}
where $B=(b^{i}_{j})$. In the following, for simplicity of notation we denote
 $$p_k=\frac{\s_k}{C_{n-1}^k}.$$

\begin{lemm} \label{lem} For $\Lambda\in\Gamma_k^+$, we have the following Newton-MacLaurin inequalities
\begin{align}
  &\frac{p_{k-1}p_{k+1}}{p_k^2}\leq1 \label{eq21},\\
    &\frac{p_1p_{k-1}}{p_k}\geq 1.\label{eq22}
  \end{align}
Moreover,  equality holds in \eqref{eq21} or \eqref{eq22} at $\Lambda$ if and only if $\Lambda=c(1,1,\cdots,1)$.
\end{lemm}
The Newton-MacLaurin inequalities play a very important role in proving geometric inequalities mentioned above. However, we will see that these inequalities are not precise enough to show our
inequality (\ref{eq8}).

 Let $\H^n=\R^+\times \S^{n-1}$ with the hyperbolic metric
\[\bar g=dr^2+\sinh ^2 r g_{\S^{n-1}},\]
where $g_{\S^{n-1}}$ is the standard round metric on the unit sphere $\S^{n-1}$
and $\Sigma \subset \H^n$ a smooth closed hypersurface in $\H^n$ with a unit outward normal $\nu$.
Let $h$ be the second fundamental form of $\Sigma$ and $\kappa=(\kappa_1,\cdots,\kappa_{n-1})$  the set of principal curvatures of $\Sigma$ in $\H^n$
with respect to $\nu$.
The $k$-th mean curvature
of $\Sigma$ is defined by
\[\s_k=\s_k(\kappa).\]

We now  consider  the following curvature evolution equation
\begin{equation}\label{flow_g}
    \frac{d}{dt} X=F\nu,
\end{equation}
where $\Sigma_t=X(t,\cdot)$ is a family of hypersurfaces in $\H^n$,
 $\nu$ is the unit outward normal to $\Sigma_t=X(t,\cdot)$ and $F$ is a speed function which may depend on the position vector $X$ and  principal curvatures of $\Sigma_t$.  One can check that \cite{Reilly}
 along the flow
 \begin{align}\label{var}
    \frac d{dt}\int_{\Sigma}\sigma_k d\mu=&(k+1)\int_{\Sigma}F\sigma_{k+1}d\mu+(n-k)\int_{\Sigma}F\sigma_{k-1}d\mu,
\end{align}
and thus
\begin{equation}\label{dpk}
\frac{d}{dt}\int_{\Sigma} p_kd\mu=\int_{\Sigma}\big((n-k-1)p_{k+1}+kp_{k-1}\big)Fd\mu .
\end{equation}
 If one compares  flow (\ref{flow_g}) in $\H^n$ with a similar flow of hypersurfaces in $\R^n$, the last term in (\ref{var}) is an extra term. This extra term comes from $-1$, the  sectional curvature of $\H^n$ and makes the phenomenon of $\H^n$ quite different from the one of $\R^n$.

 As mentioned above we use the following inverse flow
 \begin{equation}\label{flow1}
    \frac{d}{dt}X=\frac{p_{2k-1}}{p_{2k}}\nu.
\end{equation}

By using the result of Gerhardt \cite{Gerhardt} we have
\begin{prop}
If the initial hypersurface $\Sigma$ is horospherical convex, then the solution for the flow \eqref{flow1} exists for all time $t>0$ and preservs the condition of horospherical convexity. Moreover, the hypersurfaces $\Sigma_t$ become  more and more  umbilical in the sense of
\begin{equation*}
    |{h_i}^j-{\delta_i}^j|\leq Ce^{-\frac t{n-1}},\quad t>0,
\end{equation*}
i.e., the principal curvatures are uniformly bounded and converge exponentially fast to one. Here
 ${h_i}^j=g^{ik}h_{kj}$, where $g$ is the induced metric and $h$ is the second fundamental form.
 \end{prop}
\begin{proof}
 For the long time existence of the inverse curvature flow, see the work of Gerhardt \cite{Gerhardt}. The preservation of the horospherical convexity along flow (\ref{flow1}) was proved in \cite{GWW_AF2} with the help of a maximal principle for tensors of Andrews \cite{An} .
\end{proof}

Let $g$ be a Riemannian metric on $M^{n-1}$.
Denote $Ric_g$ and $R_g$  the Ricci tensor and
the scalar curvature of $g$ respectively.  The Schouten tensor $A_g$ is defined
by (\ref{Schouten}).The $\s_k$-scalar curvature, which
is introduced by Viaclovsky \cite{Via}, is defined by
\[\s_k(g):=\s_k (A_g).\]
This is a natural generalization of the scalar curvature $R$. In fact, $\s_1(g)=\frac 1{2(n-2)} R$.
Recall that $M$ is of dimension $n-1$. We now consider the conformal class $[g_{\S^{n-1}}]$ of the standard sphere $\S^{n-1}$ and
the following  functionals defined by
\begin{equation}\label{func2}
{ \mathcal F}_k(g)=vol(g)^{-\frac{n-1-2k} {n-1}}\int_{\S^{n-1}} \s_k(g)\, d\mu, \quad
k=0,1,...,n-1.
\end{equation}
 If a metric $g$ satisfies $\s_j(g)>0$ for any $j\le k$, we call it $k$-positive and denote $g\in \Gamma_k^+$.
  From Theorem 1.A in \cite{GuanWang} we have
\begin{prop} Let $0 < k < \frac{n-1} 2$ and $g\in [g_{\S^{n-1}}]$ $k$-positive. We have
\begin{equation}
\label{Sk}
{\mathcal F}_k(g)\ge {\mathcal F}_k(g_{\S^{n-1}})=\frac{C_{n-1}^k}{2^k}\omega_{n-1}^{\frac{2k}{n-1}}.
\end{equation}
\end{prop}
Inequality (\ref{Sk}) is  a generalized Sobolev inequality, since when $k=1$ inequality \eqref{Sk} is just the optimal Sobolev inequality. See for example \cite{Hebey}. For another Sobolev inequalities, see also \cite{Beckner} and \cite{ChangYang}.

\section{Relationship between various geometric quantities}

The Gauss-Bonnet curvatures $L_k$, and hence $\int_\Sigma L_k$ are intrinsic geometric quantities, which depend only on the induced metric $g$ on $\Sigma$ and
do not depend on the embeddings of $(\Sigma, g)$. Lemma  \ref{lem_a1} and Lemma \ref{lem2} below
imply that $\sigma_{2k},$ $\int \sigma_{2k}$ and $ W_{2k+1}$ are also  intrinsic.
$\sigma_{2k+1},$ $\int \sigma_{2k+1}$ and $ W_{2k}$ are extrinsic.
The  functionals $\int_\Sigma L_k$ are new geometric quantities for  the study of the integral geometry in $\H^n$.
In this section we present the relationship between
these geometric quantities.

We first have a relation between $L_k$ and $\sigma_k$.

\begin{lemm}
\label{lem0} For a hypersurface  $(\Sigma,g)$ in $\H^n$, its Gauss-Bonnet curvature $L_k$ can be expressed by higher order mean curvatures
\begin{eqnarray}\label{Lk1}
L_k&=&C_{n-1}^{2k}(2k)!\sum_{i=0}^kC_k^i(-1)^ip_{2k-2i}.
\end{eqnarray}
Hence we have
\begin{eqnarray}\label{Lk1.1}
\int_\Sigma L_k&=&C_{n-1}^{2k}(2k)!\sum_{i=0}^kC_k^i(-1)^i \int_\Sigma p_{2k-2i}=
C_{n-1}^{2k}(2k)!\sum_{i=0}^k(-1)^i \frac{C_k^i}{C_{n-1}^{2k-2i}} \int_\Sigma \sigma_{2k-2i}
.
\end{eqnarray}

\end{lemm}
\begin {proof}
 First recall the Gauss formula $${R_{ij}}^{kl}=({h_i}^k{h_j}^l-{h_i}^l{h_j}^k)-({\delta_i}^k{\delta_j}^l-{\delta_i}^l{\delta_j}^k),$$ where ${h_i}^j:=g^{ik}h_{kj}$ and $h$ is the second fundamental form.
Then substituting the Gauss formula above into (\ref{Lk}) and recalling (\ref{sigmak}), we have by a straightforward calculation,
\begin{eqnarray*}
L_k&=&\frac{1}{2^k}\d^{i_1i_2\cdots i_{2k-1}i_{2k}}
_{j_1j_2\cdots j_{2k-1}j_{2k}}{R_{i_1i_2}}^{j_1j_2}\cdots
{R_{i_{2k-1}i_{2k}}}^{j_{2k-1}j_{2k}}\\
&=&\d^{i_1i_2\cdots i_{2k-1}i_{2k}}
_{j_1j_2\cdots j_{2k-1}j_{2k}}({h_{i_1}}^{j_1}{h_{i_2}}^{j_2}-{\delta_{i_1}}^{j_1}{\delta_{i_2}}^{j_2})\cdots ({h_{i_{2k-1}}}^{j_{2k-1}}{h_{i_{2k}}}^{j_{2k}}-{\delta_{i_{2k-1}}}^{j_{2k-1}}{\delta_{i_{2k}}}^{j_{2k}})\\
&=&\sum_{i=0}^k C_k^i(-1)^i (n-2k)(n-2k+1)\cdots(n-1-2k+2i)\big((2k-2i)!\s_{2k-2i}\big)\\
&=&C_{n-1}^{2k}(2k)!\sum_{i=0}^kC_k^i(-1)^ip_{2k-2i}.
\end{eqnarray*}
Here in the second equality we use the symmetry of generalized Kronecker delta and  in the third equality we use (\ref{sigmak}) and the basic property of generalized Kronecker delta
\beq\label{Kroneckerpro.}
\d^{i_1i_2\cdots i_{p-1}i_{p}}
_{j_1j_2\cdots j_{p-1}j_{p}}{\d_{i_1}}^{j_1}=(n-p)\d^{i_2i_3\cdots i_p}
_{j_2j_3\cdots j_p},
\eeq
which follows from the Laplace expansion of determinant.

\end{proof}

Motivated by the expression (\ref{Lk1}), we introduce the following notations,
\begin{equation}\label{widetildeL}
\widetilde{L}_{k}=\sum_{i=0}^kC_k^i(-1)^ip_{2k-2i},\qquad\widetilde{N}_{k}=\sum_{i=0}^kC_k^i(-1)^ip_{2k-2i+1}.
\end{equation}
It is clear that
\[ L_k = (2k)!C_{n-1}^{2k} \widetilde{L}_{k}, \qquad N_k = (2k)!C_{n-1}^{2k} \widetilde{N}_{k}.\]

\begin{lemm} \label{lem_a1} We have
\begin{equation}\label{relation}
\sigma_{2k}=C_{n-1}^{2k}p_{2k}=C_{n-1}^{2k}\bigg (\sum_{i=0}^k C_k^i \widetilde L_{i}\bigg),
\end{equation}
and hence
\[\int_\Sigma
\sigma_{2k}=C_{n-1}^{2k}\sum_{i=0}^k C_k^i \int_\Sigma \widetilde L_{i}=\frac 1{(2k)!}\sum_{i=0}^k C_k^i  \int _\Sigma  L_{i}.
\]
\end{lemm}

To show Theorem \ref{Wkestimate} below, we need
\begin{lemm}\label{lem2}
The quermassintegral $W_{2k+1}$ can be expressed   in terms of integral of  $\widetilde L_i$
\begin{equation}\label{aimk}
W_{2k+1}(\Omega)=\frac 1n\sum_{i=0}^kC_k^i\frac{n-1-2k}{n-1-2k+2i}\int_{\Sigma}\widetilde L_{k-i}.
\end{equation}
\end{lemm}
\begin{proof}We use the induction argument to show (\ref{aimk}). When $k=0$, we have by (\ref{relation1}) that $W_1(Q)=\frac{1}{n}|\Sigma|.$  We then assume that (\ref{aimk}) holds for $k-1$, that is
\begin{eqnarray}\label{k-1}
W_{2k-1}(\Omega)&=&\frac 1n\sum_{j=0}^{k-1}C_{k-1}^j\frac{n+1-2k}{n+1-2k+2j} \widetilde L_{k-1-j}\nonumber\\
&=&\frac 1n\int_{\Sigma}\sum_{i=1}^{k}C_{k-1}^{i-1}\frac{n+1-2k}{n-1-2k+2i}\widetilde L_{k-i}.
\end{eqnarray}
By (\ref{relation1}) and (\ref{relation}), we have
\begin{eqnarray*}
W_{2k+1}(\Omega)&=&\frac{1}{n}\int_{\Sigma}p_{2k}-\frac{2k}{n-2k+1}W_{2k-1}(\Omega)\\
&=&\frac{1}{n}\int_{\Sigma}\sum_{i=0}^kC_k^i\widetilde L_i-\frac{2k}{n-2k+1}W_{2k-1}(\Omega).
\end{eqnarray*}
Substituting (\ref{k-1}) into above, one  immediately obtains (\ref{aimk}) for $k$. Thus we complete the proof.
\end{proof}

One can also show the following relation between the quermassintegrals and the curvature integrals.
\begin{lemm}\label{lem_a2}
\beq
\label{SS}
W_{2k+1}(\Omega)=\frac 1n\sum_{j=0}^k(-1)^j\frac{(2k)!!(n-2k-1)!!} {(2k-2j)!!(n-2k-1+2j)!!}\frac 1{C^{2k-2j}_{n-1}}\int_{\Sigma} \s_{2k-2j},
\eeq
 where $$(2k-1)!!:=(2k-1)(2k-3)\cdots 1\quad \mbox{and}\quad (2k)!!:=(2k)(2k-2)\cdots 2.$$
\end{lemm}
\begin{proof} One can show this relation by a direct computation. See also \cite{Santalo} or  \cite{Solanes_Thesis}.
\end{proof}

\section{Monotonicity}

In this section we prove the monotonicity of  functional $Q$ under inverse curvature flow. First, we have the variational formula
for $\int \widetilde L_k$.

\begin{lemm}
Along the inverse flow (\ref{flow1}), we have
\begin{equation}\label{aim-1}
\frac{d}{dt}\int_{\Sigma}\widetilde{ L}_k=(n-1-2k)\int\widetilde{L}_k+(n-1-2k)\int_{\Sigma}\Big(\widetilde{N}_k \frac{p_{2k-1}}{p_{2k}}-\widetilde{L}_k\Big).
\end{equation}
\end{lemm}
\begin{proof}
It follows from (\ref{dpk}) that along the inverse flow (\ref{flow_g}), we have
\begin{align}
\frac{d}{dt}\int_{\Sigma}\widetilde{ L}_k =&\int_{\Sigma} \sum_{i=0}^kC_k^i(-1)^i\Big(\big(n-1-2k+2i\big)p_{2k-2i+1}+2(k-i)p_{2k-2i-1}\Big)F\nonumber\\
=&\int_{\Sigma} \sum_{i=0}^kC_k^i(-1)^i\big(n-1-2k+2i\big)p_{2k-2i+1}F+\int\sum_{j=1}^{k}C_k^{j-1}(-1)^{j-1}2(k-j+1)p_{2k-2j+1}F\nonumber\\
=&\int_{\Sigma} \sum_{i=0}^kC_k^i(-1)^i\big(n-1-2k\big)p_{2k-2i+1}F
+\int_{\Sigma}\sum_{j=1}^{k}2(-1)^j\Big(C_k^{j}j-C_k^{j-1}(k-j+1)\Big)p_{2k-2j+1}F\nonumber\\
=&(n-1-2k)\int_{\Sigma} \sum_{i=0}^kC_k^i(-1)^ip_{2k-2i+1}\nonumber \\
=&  (n-1-2k)\int_{\Sigma} \widetilde{N}_k F \nonumber\\
=&(n-1-2k)\int_{\Sigma}\widetilde{L}_k+(n-1-2k)\int\Big(\widetilde{N}_k F-\widetilde{L}_k\Big).\nonumber
\end{align}
Substituting $F=\frac{p_{2k-1}}{p_{2k}}$ into above, we get the desired result.
\end{proof}

In order to show the monotonicity of the functional $Q$ defined in (\ref{func}) under the inverse flow \eqref{flow1}, we need to show the non-positivity of the last term in \eqref{aim-1}. That is
\beq\label{eq_x}
\frac{p_{2k-1}}{p_{2k}} \widetilde N_k-\widetilde L_k\le 0.
\eeq
When $k=1$, \eqref{eq_x} is just
\[\frac{p_1}{p_2}(p_3-p_1)-(p_2-1)\le 0,\]
which follows from the Newton-Maclaurin inequalities in Lemma \ref{lem}. In fact, it is clear that
\[\frac{p_1}{p_2}(p_3-p_1)-(p_2-1)=(\frac {p_1p_3}{p_2}-p_2)+(1-\frac{p_1^2}{p_2}).\]
Hence the non-positivity follows, for both terms are non-positive, by Lemma \ref{lem}. This was used in \cite{LWX}. When $k\ge 2$, the proof of \eqref{eq_x} becomes more
complicated. When $k=2$, one needs to show the non-positivity of
\begin{equation}\label{eq_x2}
\frac{p_3}{p_4}(p_5-2p_3+p_1)-(p_4-2p_2+1) =
\bigg(\frac{p_3}{p_4}p_5-p_4\bigg)+2\bigg(p_2-\frac{p_3^2}{p_4}\bigg)+\bigg(\frac{p_3}{p_4}p_1-1\bigg).
\end{equation}
By Lemma \ref{lem}, the first two terms are non-positive, but the last term is non-negative. It was showed in \cite{GWW_AF2} that
\eqref{eq_x2} is non-positive if $\kappa\in \R^{n-1}$ satisfying
\begin{equation}\label{h-convex}
\kappa\in \{\kappa=(\kappa_1,\kappa_2,\cdots,\kappa_{n-1})\in\R^{n-1}\,|\, \kappa_i\ge 1\}.
\end{equation}
We want to show that (\ref{eq_x}) is true for general $k\leq\frac 12(n-1)$.
This is one of key points of this paper. Now the case is more complicated than the case $k=2$.

\begin{prop}\label{RefinedAF}For any $\kappa$ satisfying (\ref{h-convex}), we have 
\begin{align}\label{aim}
\frac{p_{2k-1}}{p_{2k}}\widetilde{N}_k-\widetilde{L}_k\leq0.
\end{align}
Equality holds if and only if
one of the following two cases holds
\beq\label{=}
\hbox{either } \quad (i)\,  \kappa_i=\kappa_j \; \forall \, i,j,  \quad\hbox{ or }\quad (ii)\, \exists \, i\, \hbox{ with }\kappa_i <1\, \& \, \kappa_j=1 \; \forall j\neq i.
\eeq
\end{prop}

We sketch the proof into several steps. Before the proof,  we introduce the notation of $\sum\limits_{cyc}$ to simplify notations. Precisely,
given $n-1$ numbers $(\kappa_1,\kappa_2,\cdots, \kappa_{n-1})$, we denote $\sum\limits_{cyc}f(\kappa_1,\cdots,\kappa_{n-1})$
the cyclic summation which takes over all {\it different} terms of the type $f(\kappa_1,\cdots,\kappa_{n-1})$.
For instance,
\begin{align*}
&\sum_{cyc}\kappa_1=\kappa_1+\kappa_2+\cdots+\kappa_{n-1},\qquad\sum_{cyc}\kappa_1^2\kappa_2=\sum_{i=1}^{n-1}\Big(\kappa_i^2\sum_{j\neq i}\kappa_j\Big),\\
&\sum_{cyc}\kappa_1(\kappa_2-\kappa_3)^2=\sum_{i=1}^{n-1}\bigg(\kappa_i\sum_{\substack {1\leq j<k\leq n-1\\j,k\neq i}}(\kappa_j-\kappa_k)^2\bigg),\\
&\qquad\qquad\qquad\quad\;={(n-3)}\sum_{cyc}\kappa_1\kappa_2^2-6\sum_{cyc}\kappa_1\kappa_2\kappa_3.
\end{align*}

\begin{lemm}\label{key lemm}
 For any $\kappa$ satisfying (\ref{h-convex}), we have
\begin{equation}\label{key inequ.}
\widetilde{N}_k-p_1\widetilde{L}_k\leq0.
\end{equation}
Equality holds if and only if
one of the following two cases holds
$$
\hbox{either } \quad (i)\,  \kappa_i=\kappa_j \; \forall \, i,j,  \quad\hbox{ or }\quad (ii)\, \exists \, i\, \hbox{ with }\kappa_i >1\, \& \, \kappa_j=1 \; \forall j\neq i.$$
\end{lemm}
\begin{proof}
It is crucial to observe that (\ref{key inequ.}) is indeed equivalent to the following inequality:
\begin{align}\label{permu-sum-3}
\sum_{1\le i_m\le n-1, i_j\neq i_l(j\neq l)} \kappa_{i_1}(\kappa_{i_2}\kappa_{i_3}-1)
(\kappa_{i_4}\kappa_{i_5}-1)\cdots(\kappa_{i_{2k-2}}\kappa_{i_{2k-1}}-1)\big(\kappa_{i_{2k}}-\kappa_{i_{2k+1}}\big)^2 \ge 0,
\end{align}
where the summation takes over all the  $(2k+1)$-elements permutation of $\{1,2,\cdots, n-1\}$.
For the convenience of the reader, we sketch the proof of (\ref{permu-sum-3}) briefly.
First, note that from (\ref{widetildeL}) that
\begin{eqnarray*}
(p_1\widetilde{L}_k-\widetilde{N}_k)
&=&p_1\sum_{i=0}^k C_k^{k-i}(-1)^{k-i}p_{2i}-\sum_{i=0}^k C_k^{k-i}(-1)^{k-i}p_{2i+1}\\
&=&\sum_{i=0}^k (-1)^{k-i}C_k^i (p_1p_{2i}-p_{2i+1}).
\end{eqnarray*}
Next we calculate each term $p_1p_{2i}-p_{2i+1}$ carefully.  By using $$(n-1)C_{n-1}^j=(j+1)C_{n-1}^{j+1}+(n-1)C_{n-1}^{j-1}),$$  we have
\begin{eqnarray*}
\sigma_j \sigma_1&=&\bigg(\sum_{cyc} \kappa_{i_1}\kappa_{i_2}\cdots \kappa_{i_{j}}\bigg)\bigg(\sum_{cyc} \kappa_{i_{j+1}}\bigg)\\
&=&(j+1)\bigg(\sum_{cyc} \kappa_{i_1}\kappa_{i_2}\cdots \kappa_{i_{j}}\kappa_{i_{j+1}}\bigg)+\sum_{cyc} \kappa_{i_1}^2\kappa_{i_2}\cdots \kappa_{i_{j}},
\end{eqnarray*}
and
\begin{eqnarray*}
&&p_1p_{2j}-p_{2j+1}\\
&=&\frac{1}{(n-1)C_{n-1}^{2j}}\bigg(\sum_{cyc} \kappa_{i_1}\kappa_{i_2}\cdots \kappa_{i_{2j}}\bigg)\bigg(\sum_{cyc} \kappa_{i_{2j+1}}\bigg)-
 \frac{1}{C_{n-1}^{2j+1}}\sum_{cyc} \kappa_{i_1}\kappa_{i_2}\cdots \kappa_{i_{2j+1}}\\
&=&\frac{1}{(n-1)C_{n-1}^{2j}C_{n-1}^{2j+1}}\Big(C_{n-1}^{2j+1}(2j+1)\sum_{cyc} \kappa_{i_1}\kappa_{i_2}\cdots \kappa_{i_{2j}}\kappa_{i_{2j+1}}
+C_{n-1}^{2j+1}\sum_{cyc} \kappa_{i_1}^2\kappa_{i_2}\cdots \kappa_{i_{2j}}\\
&&\qquad-(n-1)C_{n-1}^{2j}\sum_{cyc} \kappa_{i_1}\kappa_{i_2}\cdots \kappa_{i_{2j+1}}\Big)\\
&=&\frac{1}{(n-1)C_{n-1}^{2j}C_{n-1}^{2j+1}}\cdot\frac{C_{n-1}^{2j+1}}{n-2j}\sum_{cyc} \kappa_{i_1}\kappa_{i_2}\cdots \kappa_{i_{2j-1}}(\kappa_{i_{2j}}-\kappa_{i_{2j+1}})^2\\
&=&\frac{(2j)!(n-2j-2)!}{(n-1)\cdot (n-1)!}\sum_{cyc} \kappa_{i_1}\kappa_{i_2}\cdots \kappa_{i_{2j-1}}(\kappa_{i_{2j}}-\kappa_{i_{2j+1}})^2.
\end{eqnarray*}
In (\ref{permu-sum-3}), the coefficient of $\kappa_{1}\kappa_{2}\cdots \kappa_{2j-1}(\kappa_{2j}-\kappa_{2j+1})^2$ is
\begin{eqnarray*}
&&2(-1)^{k-j}C_{k-1}^{j-1}{(2j-1)!}C_{n-2j-2}^{2k-2j}{[2(k-j)]!}=\frac{(-1)^{k-j}}{k}C_k^j(2j)!(n-2j-2)!\\
&&\qquad=(-1)^{k-j}C_k^j\frac{(2j)!(n-2j-2)!}{(n-1)\cdot (n-1)!}\cdot\frac{ (n-1)\cdot (n-1)!}{k}.
\end{eqnarray*}
Therefore we have
\begin{align*}
0&\le \sum_{1\le i_m\le n-1, i_j\neq i_l(j\neq l)} \kappa_{i_1}(\kappa_{i_2}\kappa_{i_3}-1)
(\kappa_{i_4}\kappa_{i_5}-1)\cdots(\kappa_{i_{2k-2}}\kappa_{i_{2k-1}}-1)\big(\kappa_{i_{2k}}-\kappa_{i_{2k+1}}\big)^2 \\
&=\frac{ (n-1)\cdot (n-1)!}{k}\sum_{j=0}^k (-1)^{k-j}C_k^j (p_1p_{2j}-p_{2j+1})\\
&=\frac{(n-1)\cdot (n-1)!}{k}(p_1\widetilde{L}_k-\widetilde{N}_k).
\end{align*} This finishes  the proof.
\end{proof}

In view of (\ref{permu-sum-3}), we have the following remark which will be used later.
\begin{rema}\label{remark:2} For any $\kappa=(\kappa_1,\cdots,\kappa_{n-1})$ satisfying $0<\kappa_i\leq 1,\quad (i=1,\cdots,n-1)$,  then $(-1)^{k-1}\big(\widetilde{N}_k-p_1\widetilde{L}_k\big)\le0$.
\end{rema}
\begin{lemm}\label{lemm2}
For any $\kappa$ satisfying (\ref{h-convex}), we have
$$\widetilde{N}_k \geq0, \quad\widetilde{L}_k\geq0.$$
\end{lemm}
\begin{proof}
They are equivalent to the following inequalities respectively:
\begin{align}\label{permu-sum-1}
&\sum_{1\le i_m\le n-1, i_j\neq i_l(j\neq l)} \kappa_{i_1}(\kappa_{i_2}\kappa_{i_3}-1)
(\kappa_{i_4}\kappa_{i_5}-1)\cdots(\kappa_{i_{2k-2}}\kappa_{i_{2k-1}}-1)(\kappa_{i_{2k}}\kappa_{i_{2k+1}}-1) \ge 0,\\
&\sum_{1\le i_m\le n-1, i_j\neq i_l(j\neq l)} (\kappa_{i_2}\kappa_{i_3}-1)
(\kappa_{i_4}\kappa_{i_5}-1)\cdots(\kappa_{i_{2k-2}}\kappa_{i_{2k-1}}-1)(\kappa_{i_{2k}}\kappa_{i_{2k+1}}-1) \ge 0.\label{permu-sum-2}
\end{align}
where the summation takes over all the  $(2k+1)$-elements permutation of $\{1,2,\cdots, n-1\}$.
The proof to show the equivalence of (\ref{permu-sum-1}),(\ref{permu-sum-2}) is exactly
the same as the one of (\ref{permu-sum-3}). Hence we omit it here.
\end{proof}
\begin{rema}
For any $\kappa=(\kappa_1,\cdots,\kappa_{n-1})$ satisfying $0<\kappa_i\leq 1,\quad (i=1,\cdots,n-1)$, then $(-1)^k\widetilde{N}_k \geq0, (-1)^k\widetilde{L}_k\geq0$.
\end{rema}

Making use of Lemma \ref{key lemm} and Remark \ref{remark:2},  we can show the following result which is stronger than Proposition \ref{RefinedAF}.
\begin{lemm} For any $\kappa$ satisfying (\ref{h-convex}), we have
 $$p_{2k}\widetilde{N}_k-p_{2k+1}\widetilde{L}_k\le0.$$
\end{lemm}
\begin{proof} According to the induction argument proved in \cite{GWW_AF2} (see p.8),
we only need to prove it for $n-1=2k+1$.
Let $z_i=\frac1{\kappa_i}\leq1$, and
$$
\hat{p}_i=p_i(z_1, z_2, \cdots, z_{2k+1}).
$$
It is clear that
\begin{align}\label{hatp}
\hat{p}_{j}=\frac{p_{2k+1-j}}{p_{2k+1}}.
\end{align}
By Remark \ref{remark:2}, we have
\begin{align}
(-1)^{k-1}\sum_{i=0}^kC_k^i(-1)^i\hat{p}_{2k-2i+1}-(-1)^{k-1}\hat{p}_1\sum_{i=0}^kC_k^i(-1)^i\hat{p}_{2k-2i}\le0,
\end{align}
which  is equivalent to
\begin{align}
(-1)^{k-1}\sum_{i=0}^kC_k^i(-1)^i\frac{p_{2i}}{p_{2k+1}}-(-1)^{k-1}\frac{{p}_{2k}}{p_{2k+1}}\sum_{i=0}^kC_k^i(-1)^i\frac{{p}_{2i+1}}{p_{2k+1}}\le0.
\end{align}
Thus we have
\begin{align}
\sum_{i=0}^kC_k^i(-1)^{k-i}{p_{2i}}-\frac{{p}_{2k}}{p_{2k+1}}\sum_{i=0}^kC_k^i(-1)^{k-i}{{p}_{2i+1}}\ge0,
\end{align}
which implies $\frac{p_{2k}}{p_{2k+1}}\widetilde{N}_k-\widetilde{L}_k\le0.$
\end{proof}
\vspace{2mm}
\noindent{\it Proof of Proposition \ref{RefinedAF}.}
Then by the Newton-MacLaurin inequality $p_{2k-1}p_{2k+1}\le p_{2k}^2$, we obtain
$$\frac{p_{2k-1}}{p_{2k}}\widetilde{N}_k-\widetilde{L}_k\leq\frac{p_{2k}}{p_{2k+1}}\widetilde{N}_k-\widetilde{L}_k\le0,$$
which is exactly (\ref{aim}). Here we have used Lemma \ref{lemm2}.
\qed

\begin{rema}\label{rk1} Proposition \ref{RefinedAF} holds for $\kappa\in \R^{n-1}$ with $\kappa_i\kappa_j\ge 1$ for any $i,j$. This is equivalent to
the condition that the sectional curvature of $\Sigma$ is non-negative.
\end{rema}

\begin{rema}\label{rk2}
 From the proof of Proposition \ref{RefinedAF},  it is easy to see that
\eqref{aim} has an inverse inequality for $\kappa\in \R^{n-1}$ with $0\le \kappa_i\le 1$.
\end{rema}

Now we have a monotonicity of $Q(\Sigma_t)$ defined by \eqref{func}  under the flow \eqref{flow1}.
\begin{thm}\label{thm1} Functional $Q$ is non-increasing under  the flow \eqref{flow1}, provided that the initial surface is horospherical convex.
\end{thm}
\begin{proof}
It follows from (\ref{Lk1}), (\ref{widetildeL}) and  Proposition \ref{RefinedAF} that
\begin{equation}\label{Lkevolve}
\frac{d}{dt}\int_{\Sigma}L_k\leq (n-1-2k)\int_{\Sigma}L_k.
\end{equation}
On the other hand, by (\ref{dpk}) and (\ref{eq22}), we also have
\begin{equation}\label{area}
\frac{d}{dt}|\Sigma_t|=\int_{\Sigma_t}\frac{p_{2k-1}}{p_{2k}}(n-1)p_1d\mu\geq (n-1)|\Sigma_t|.
\end{equation}
Combining (\ref{Lkevolve}) and (\ref{area}) together, we complete the proof.

\end{proof}

\begin{rema}
From the above proof, one can check that  to obtain a monotonicity of $Q$ it is enough to choose $F=\frac{1}{p_1}$. Then from (\ref{aim-1}) and (\ref{key inequ.}), it holds for all $k$
\begin{align}\nonumber
\frac{d}{dt}\int\widetilde{ L}_k =&(n-2k-1) \int\widetilde{ L}_k+(n-2k-1)\Big(\frac{1}{p_1}\widetilde{N}_k-\widetilde{L}_k\Big)\\
\le &(n-2k-1) \int\widetilde{ L}_k.\nonumber
\end{align}
\end{rema}
\section{Proof of main Theorems}

Now we are ready to show our main theorems.

\

\noindent{\it Proof of Theorem \ref{mainthm1}.} First recall the definition (\ref{func}) of the functional $Q$ , (\ref{eq8}) is equivalent to
\begin{equation}\label{ineq2} Q(\Sigma)\ge C_{n-1}^{2k}(2k)!\;\omega_{n-1}^{\frac{2k}{n-1}}.
\end{equation}
 Let $\Sigma(t)$ be a solution of flow (\ref{flow1}) obtained by the work of Gerhardt \cite{Gerhardt}. This flow preserves the horospherical
convexity and non-increases for the functional $Q$. Hence, to show \eqref{ineq2} we only need  to show
\begin{equation}
\lim_{t\to \infty} Q(\Sigma_t)\ge  C_{n-1}^{2k}(2k)!\;\omega_{n-1}^{\frac{2k}{n-1}}.
\end{equation}
Since $\Sigma$ is a horospherical convex hypersurface in $(\H^n,\bar g)$, it is written  as graph of function $r(\theta)$, $\theta\in \mathbb{S}^{n-1}$. We denote $X(t)$ as graphs $r(t,\theta)$ on $\mathbb{S}^{n-1}$ with the standard metric $\hat g$. We set $\lambda(r)=\sinh(r)$ and we have $\lambda'(r)=\cosh(r)$. It is clear that
$$
(\lambda')^2=(\lambda)^2+1.
$$
We define $\varphi(\theta)=\Phi(r(\theta))$. Here $\Phi$ verifies
$$
\Phi'=\frac{1}{\lambda}.
$$
We define another function
$$
v=\sqrt{1+|\nabla\varphi|^2_{\hat g}}.
$$
By \cite{Gerhardt}, we have the following results.\\

\begin{lemm}\label{Ger.}
$$
\lambda=O(e^{\frac{t}{n-1}}),\qquad |\nabla\varphi|+|\nabla^2\varphi|=O(e^{-\frac{t}{n-1}}).
$$
\end{lemm}
From Lemma (\ref{Ger.}), we have the following expansions:
\beq\label{lambda'}
\lambda'=\lambda(1+\frac12\lambda^{-2})+O(e^{-\frac{4t}{n-1}}),
\eeq
and
\beq\label{vinverse}
\frac 1v=1-\frac12|\nabla \varphi|^2_{\hat g}+O(e^{-\frac{4t}{n-1}}).
\eeq

We have also
\beq\label{nablalambda}
\nabla \lambda=\lambda\lambda'\nabla\varphi.
\eeq

The
second fundamental form of $\Sigma$ is written in an orthogonal basis (see \cite{Ding} for example)
$$
\begin{array}{lll}
{h_i}^j&=&\ds\frac{\lambda'}{v\lambda}\left({\delta_i}^j-\frac{{\varphi_i}^j}{\lambda'}+\frac{{\varphi_i}{\varphi_l}\varphi^{jl}}{v^2\lambda'}\right) \\
&=&\ds {\delta_i}^j+(\frac{1}{2\lambda^2}-\frac12|\nabla \varphi|^2){\delta_i}^j-\frac{{\varphi_i}^j}{\lambda}+O(e^{-\frac{4t}{n-1}}),
\end{array}
$$
where the second equality follows from (\ref{lambda'}) and (\ref{vinverse}).
We set
\beq\label{def.T}
{T_i}^j=(\frac{1}{2\lambda^2}-\frac12|\nabla \varphi|^2){\delta_i}^j-\frac{{\varphi_i}^j}{\lambda},
\eeq
then from the Gauss equations, we obtain
\begin{eqnarray*}
{R_{ij}}^{kl}&=&-({\delta_i}^k{\delta_j}^l-{\delta_i}^l{\delta_j}^k)+({h_i}^k{h_j}^l-{h_i}^l{h_j}^k)\\
&=&{\delta_i}^k{T_j}^l+{T_i}^k{\delta_j}^l-{T_i}^l{\delta_j}^k-{\delta_i}^l{T_j}^k+O(e^{-\frac{4t}{n-1}}).
\end{eqnarray*}
It follows from (\ref{Lk}) that
$$
\begin{array}{lllc}
L_k&=&\frac{1}{2^k}\d^{i_1i_2\cdots i_{2k-1}i_{2k}}
_{j_1j_2\cdots j_{2k-1}j_{2k}}{R_{i_1i_2}}^{j_1j_2}\cdots
{R_{i_{2k-1}i_{2k}}}^{j_{2k-1}j_{2k}}\\
&=&2^k \d^{i_1i_2\cdots i_{2k-1}i_{2k}}
_{j_1j_2\cdots j_{2k-1}j_{2k}}{T_{i_1}}^{j_1}{\d_{i_2}}^{j_2}\cdots
{T_{i_{2k-1}}}^{j_{2k-1}}{\d_{i_{2k}}}^{j_{2k}}+O(e^{-\frac{(2k+2)t}{n-1}})\\
&=&2^k (n-1-k)\cdots(n-2k) \d^{i_1i_3\cdots i_{2k-1}}
_{j_1j_3\cdots j_{2k-1}}{T_{i_1}}^{j_1} {T_{i_3}}^{j_3}\cdots
{T_{i_{2k-1}}}^{j_{2k-1}}+O(e^{-\frac{(2k+2)t}{n-1}})\\
&=&2^k k! (n-1-k)\cdots(n-2k)\s_k(T)+O(e^{-\frac{(2k+2)t}{n-1}}).
\end{array}
$$
Here in the second equality we use the fact $$
\begin{array}{rcl}
\d^{i_1i_2\cdots i_{2k-1}i_{2k}}
_{j_1j_2\cdots j_{2k-1}j_{2k}}{T_{i_1}}^{j_1}{\d_{i_2}}^{j_2} &=& \ds\vs  \d^{i_1i_2\cdots i_{2k-1}i_{2k}}
_{j_1j_2\cdots j_{2k-1}j_{2k}}{\d_{i_1}}^{j_1}{T_{i_2}}^{j_2} \\&=&-\d^{i_1i_2\cdots i_{2k-1}i_{2k}}
_{j_1j_2\cdots j_{2k-1}j_{2k}}{T_{i_1}}^{j_2}{\d_{i_2}}^{j_1}=- \d^{i_1i_2\cdots i_{2k-1}i_{2k}}
_{j_1j_2\cdots j_{2k-1}j_{2k}}{\d_{i_1}}^{j_2}{T_{i_2}}^{j_1}, \end{array}$$
and in the third equality we use (\ref{sigmak}) and (\ref{Kroneckerpro.}).\\
Recall $\varphi_i=\lambda_i/\lambda\lambda'$, then by (\ref{lambda'}) we have
\beq\label{varphi''}
\varphi_{ij}=\frac{\lambda_{ij}}{\lambda^2}-\frac{2\lambda_i\lambda_j}{\lambda^3}+O(e^{-\frac{3t}{n-1}}).
\eeq
By the definition of the Schouten tensor,
$$
A_{\hat g}=\frac{1}{n-3}\left(Ric_{\hat g}-\frac{R_{\hat g}}{2(n-2)}{\hat g}\right)=\frac12 \hat g.
$$
Its conformal transformation formula is well-known (see for example \cite{Via})
\begin{equation}\label{b2}
A_{\lambda^2\hat g}=-\frac{\nabla^2\lambda}{\lambda}+\frac{2\nabla\lambda\otimes\nabla\lambda }{\lambda^2} -\frac12\frac{ |\nabla\lambda|^2}{\lambda^2}\hat g+A_{\hat g}=-\frac{\nabla^2\lambda}{\lambda}+\frac{2\nabla\lambda\otimes\nabla\lambda }{\lambda^2} -\frac12\frac{ |\nabla\lambda|^2}{\lambda^2}\hat g+\frac 12 \hat g.
\end{equation}
Substituting (\ref{nablalambda}) and (\ref{varphi''}) into (\ref{def.T}), together with (\ref{b2}), we have
$$
{T_i}^j=({(\lambda^2\hat g)^{-1}A_{\lambda^2\hat g})_i}^j+O(e^{-\frac{4t}{n-1}}),
$$
which implies
\beq\label{eq3.11}
L_k=2^k k! (n-1-k)\cdots(n-2k)\s_k(A_{\lambda^2\hat g})+O(e^{-\frac{(2k+2)t}{n-1}}).
\eeq
As before, $\Sigma(t)$ is a  horospherical convex hypersurface. As a consequence, $\Sigma$  has the nonnegative sectional curvature so that
$T+O(e^{-\frac{4t}{n-1}})$ is positive definite.
We consider $\tilde \lambda:=\lambda^{1-e^{-\frac{t}{n-1}}}$ and the conformal metric $\tilde \lambda^2\hat g$. We have
\begin{eqnarray*}
\tilde\lambda^2(\tilde \lambda^2\hat g)^{-1}A_{{\tilde \lambda^2\hat g}}&=&\frac{1}{2}e^{-\frac{t}{n-1}}I+\frac{1}{2}e^{-\frac{t}{n-1}}(1-e^{-\frac{t}{n-1}})\frac{|\nabla \lambda|^2}{ \lambda^2}I-e^{-\frac{t}{n-1}}(1-e^{-\frac{t}{n-1}})\hat g^{-1}\frac{\nabla\lambda\otimes\nabla\lambda}{\lambda^2}\\
&&+\lambda^2(1-e^{-\frac{t}{n-1}})( \lambda^2\hat g)^{-1}A_{{ \lambda^2\hat g}}.
\end{eqnarray*}
Recall $\frac{1}{2}e^{-\frac{t}{n-1}}I+\lambda^2(1-e^{-\frac{t}{n-1}})( \lambda^2\hat g)^{-1}A_{{ \lambda^2\hat g}}\in  \Gamma_{n-1}^+$ for the sufficiently large $t$ and $\frac{1}{2}e^{-\frac{t}{n-1}}(1-e^{-\frac{t}{n-1}})\frac{|\nabla \lambda|^2}{ \lambda^2}I-e^{-\frac{t}{n-1}}(1-e^{-\frac{t}{n-1}})\hat g^{-1}\frac{\nabla\lambda\otimes\nabla\lambda}{\lambda^2}\in \overline{\Gamma_{k}^+}$ for any $k\le \frac{n-1}{2}$. Therefore, we infer $\tilde \lambda^2\hat g\in  \Gamma_k^+$ for any $k\le \frac{n-1}{2}$. The Sobolev inequality  (\ref{Sk}) for the $\s_k$ operator gives
\beq\label{eq3.12}
(vol(\tilde \lambda^2\hat g))^{-\frac{n-1-2k}{n-1}}\int_{\mathbb{S}^{n-1}}\s_k(A_{\tilde \lambda^2\hat g}) dvol_{\tilde \lambda^2\hat g}\ge \frac{(n-1)\cdots(n-k)}{2^kk!}\omega_{n-1}^{\frac{2k}{n-1}}.
\eeq
On the other hand, we have
\beq\label{eq3.13}\begin{array}{rcl}
\ds \vs && (vol(\tilde \lambda^2\hat g))^{-\frac{n-1-2k}{n-1}}\int_{\mathbb{S}^{n-1}}
\s_k(A_{\tilde \lambda^2\hat g}) dvol_{\tilde \lambda^2\hat g} \\ &=&(1+o(1))(vol( \lambda^2\hat g))^{-\frac{n-1-2k}{n-1}}\int_{\mathbb{S}^{n-1}}\s_k(A_{ \lambda^2\hat g}) dvol_{ \lambda^2\hat g},
\end{array}
\eeq
since
$$
\lambda^{-e^{-\frac{t}{n-1}}}=1+o(1).
$$
As a consequence of (\ref{eq3.11}),(\ref{eq3.12}) and (\ref{eq3.13}), we deduce
$$
\lim_{t\to +\infty} (vol(\Sigma(t)))^{-\frac{n-1-2k}{n-1}}\int_{\Sigma(t)} L_k\ge (n-1)(n-2)\cdots(n-2k)\omega_{n-1}^{\frac{2k}{n-1}}.
$$
When  \eqref{ineq2} is an equality, then $Q$ is constant along the flow. Then \eqref{area} is an equality, which implies that equality
 in  the inequality
 \[\frac{p_{2k-1}}{p_{2k}}p_1\ge 1,\]
 holds. Therefore,
$\Sigma$ is a geodesic sphere.
\qed

\vspace{5mm}

\noindent{\it Proof of Theorem \ref{mainthm2}.}
It follows from (\ref{Lk1}), (\ref{widetildeL}) and Theorem \ref{mainthm1} that when $n-1>2k$
\begin{equation}\label{Lklimit}
\int_{\Sigma} \tilde L_k\ge \omega_{n-1}^{\frac{2k}{n-1}}(|\Sigma|)^{\frac{n-1-2k}{n-1}}.
\end{equation}
Using the expression \eqref{relation}  of $\int_{\Sigma}\s_k$ in terms of $\int_{\Sigma} \tilde L_j$
we get the desired result

\begin{equation*}
\begin{array}{rcl}
\ds \int_\Sigma \s_{2k}\ge \ds\vs C_{n-1}^{2k}\omega_{n-1}\left\{ \left( \frac{|\Sigma|}{\omega_{n-1}} \right)^\frac 1k +
\left( \frac{|\Sigma|}{\omega_{n-1}} \right)^{\frac 1k\frac {n-1-2k}{n-1}} \right\}^k.
\end{array}
\end{equation*}
By Theorem \ref{mainthm1}, equality holds if and only if $\Sigma$ is a geodesic sphere.

When $n-1=2k$, since the hypersurface $\Sigma$ is convex.
 we know that \eqref{eq8} is an equality when $n-1=2k$ by the Gauss-Bonnet-Chern theorem, even for any hypersurface diffeomorphic to a sphere. Hence  in this case, we also have all the above inequalities
with equality  which in turn implies by \cite{LWX} or \cite{GWW_AF2} that $\Sigma$
is a geodesic sphere.
\qed
\\

\noindent{\it Proof of Theorem \ref{Wkestimate}.}
When $n-1>2k$, the proof follows directly from (\ref{Lklimit}) and Lemma \ref{lem2}. When $n-1=2k$, the proof follows by the same reason as in Theorem \ref{mainthm2}.
\qed

\

 From \eqref{relation1}, it is easy to see that Theorem \ref{Wkestimate} implies Theorem \ref{mainthm2}, meanwhile Theorem \ref{mainthm2} may not  directly imply
 Theorem \ref{Wkestimate}, since there are negative coefficients in \eqref{SS} above.

\section{Alexandrov-Fenchel inequality for odd $k$}
In this section, we show  an Alexandrov-Fenchel inequality for $\sigma_1$, which follows from the result of Cheng-Zhou \cite{CZ} and Theorem \ref{mainthm2}
(or more precisely from \cite{LWX}).

\begin{thm}
\label{thm4} Let $n\ge2$. Any horospherical convex hypersurface  $\Sigma\subset\H^n$ satisfies
\begin{equation}\label{AF1}
\int_{\Sigma}\sigma_1\geq (n-1)\omega_{n-1}\bigg\{\bigg(\frac{|\Sigma|}{\omega_{n-1}}\bigg)^2+\bigg(\frac{|\Sigma|}{\omega_{n-1}}\bigg)^{\frac{2(n-2)}{n-1}}\bigg\}^{\frac 12}.
\end{equation}
where $\omega_{n-1}$ is the area of the unit sphere $\S^{n-1}$ and $|\Sigma|$ is the area of $\Sigma$.
Equality holds if and only if $\Sigma$ is a geodesic sphere.
\end{thm}
\begin{proof}
Notice that the horospherical convex condition implies that the Ricci curvature  of $\Sigma$ is non-negative.
We observe first that by a direct computation   (1.4) in \cite{CZ}
\[\int_\Sigma |H-\overline{H}|^2 \le \frac{n-1}{n-2} \int_\Sigma |B-\frac H{n-1} g|^2,
\]
is equivalent to
\beq\label{eq4.1}
\int_{\Sigma}\sigma_2\int_{\Sigma}\sigma_0\leq\frac{n-2}{2(n-1)}\big(\int_{\Sigma}\sigma_1\big)^2.
\eeq
Then we use  the optimal inequality for $\sigma_2$  proved in \cite{LWX} (see also Theorem 1.2),
\beq\label{AF2}
\int_{\Sigma}\sigma_2\geq\frac{(n-1)(n-2)}{2}\bigg(\omega_{n-1}^{\frac{2}{n-1}}|\Sigma|^{\frac{n-3}{n-1}}+|\Sigma|\bigg),
\eeq
to  obtain the  desired inequality for $\sigma_1$,
$$\int_{\Sigma}\sigma_1\geq (n-1)\omega_{n-1}\bigg\{\bigg(\frac{|\Sigma|}{\omega_{n-1}}\bigg)^2+\bigg(\frac{|\Sigma|}{\omega_{n-1}}\bigg)^{\frac{2(n-2)}{n-1}}\bigg\}^{\frac 12}.$$
When (\ref{AF1}) is an equality, in turn, (\ref{AF2}) is also a equality, then it follows from \cite{LWX} that
 the hypersurface is a geodesic sphere.
\end{proof}

\

Motivated by Theorem \ref{mainthm2} and  (\ref{eq4.1}), we would like propose the following
\begin{conj} Let $n-1\ge2k+1$. Any horospherical convex hypersurface  $\Sigma\subset\H^n$ satisfies
$$\int_{\Sigma}\sigma_{2k+1}\geq C_{n-1}^{2k+1}\omega_{n-1}\bigg\{\bigg(\frac{|\Sigma|}{\omega_{n-1}}\bigg)^{\frac{2}{2k+1}}+\bigg(\frac{|\Sigma|}{\omega_{n-1}}\bigg)^{\frac{2}{2k+1}\frac{(n-2k-2)}{n-1}}\bigg\}^{\frac {2k+1}{2}}.$$Equality holds if and only if $\Sigma$ is a geodesic sphere.
\end{conj}
The conjecture  follows from Theorem \ref{mainthm2} and the following conjecture
\begin{equation}\label{ai}
\frac{\big(C_{n-1}^{2k+1}\big)^2}{C_{n-1}^{2k+2}C_{n-1}^{2k}}\int_{\Sigma}\sigma_{2k+2}\int_{\Sigma}\sigma_{2k}\leq \bigg(\int_{\Sigma}\sigma_{2k+1}\bigg)^2.
\end{equation}

\

\

\noindent{\it Acknowledgment.}
We would like to thank  Wei Wang for his important help
in the proof of the monotonicity of functional $Q$.

\end{document}